\documentclass[11pt]{article}

\usepackage{amsmath,amsthm}

\usepackage{amssymb,latexsym}

\usepackage{enumerate}

\newcommand{\BB}{{\cal B}}

\newcommand{\EE}{{\cal E}}

\newcommand{\FF}{{\cal F}}

\newcommand{\MM}{{\cal M}}

\newcommand{\SSS}{{\cal S}}

\newcommand{\BD}{{\mathbb D}}
\newcommand{\BN}{{\mathbb{N}}}
\newcommand{\BR}{{\mathbb{R}}}
\newcommand{\BX}{{\mathbb{X}}}
\newcommand{\dyw}{\mbox{\rm div}}
\newcommand{\fch}{\mathbf{1}}

\newtheorem{theorem}{\bf Theorem}[section]
\newtheorem{proposition}[theorem]{\bf Proposition}
\newtheorem{lemma}[theorem]{\bf Lemma}
\newtheorem{corollary}[theorem]{\bf Corollary}

\theoremstyle{definition}

\newtheorem{remark}[theorem]{Remark}
\newtheorem{example}[theorem]{Example}

\numberwithin{equation}{section}

\begin{document}

\title{On the structure of bounded smooth measures
associated with a quasi-regular Dirichlet form}
\author{Tomasz  Klimsiak \ and Andrzej Rozkosz}
\date{}
\maketitle
\begin{abstract}
We consider a quasi-regular Dirichlet form. We show that a bounded
signed measure charges no set of zero capacity associated with the
form if and only if the measure decomposed into the sum of an
integrable function and a bounded linear functional on the domain
of the form. The decomposition allows one to describe explicitly
the set of bounded measures charging no sets of zero capacity for
interesting classes of Dirichlet forms. By way of illustration,
some examples are given.
\end{abstract}



\section{Introduction}

Boccardo, Gallou\"et and Orsina \cite{BGO} have shown that if
$D\subset\BR^d$ is an open bounded set and $\mu$ is a bounded
(signed) Borel measure  on $D$ then  $\mu$ charges no set  of zero
Newtonian capacity if and only if $\mu$ can be decomposed into the
sum of an integrable function and an element of the dual space
$H^{-1}(D)$ of the Sobolev space $H^1_0(D)$. If we denote by
$\MM_b$ the space of all bounded Borel measures on $D$ with
bounded total variation, and by $\MM_{0,b}$ its subset consisting
of all measures charging no set of zero capacity, then the
decomposition of \cite{BGO} can be stated succinctly as
\begin{equation}
\label{eq1.1} \MM_{0,b}(D)=L^1(D;dx)+ H^{-1}(D)\cap \MM_{b}(D).
\end{equation}
The decomposition (\ref{eq1.1}) when combined with the analogue of the
Lebesgue decomposition theorem saying that each bounded Borel
measure on $D$ can be uniquely decomposed into the absolutely
continuous and the singular part with respect to the capacity (see
Fukushima, Sato and Taniguchi \cite{FST}) gives a complete
description of bounded Borel measures on $D$.

In the language of Dirichlet forms the decomposition (\ref{eq1.1})
says that, if $\mu\in\MM_b(D)$, then its variation $|\mu|$ is smooth
with respect to the capacity associated with the classical
Dirichlet form if and only if $\mu$ admits a decomposition into an
integrable function and an element of the dual space of the domain
of the form. The decomposition (\ref{eq1.1}) is interesting in its own
right and together with the decomposition of \cite{FST} proved to
be very useful in investigating elliptic equations involving local
operators and measure data (see, e.g., \cite{BGO,BMP,DMOP,MP}; see
also \cite{DPP} for a parabolic version of (\ref{eq1.1}) and its
applications to parabolic equations). This and the fact that the
decomposition in \cite{FST} is proved in the setting of general
Dirichlet forms motivated us to ask whether (\ref{eq1.1}) can  also 
be generalized to the case of bounded smooth measures with respect
to general Dirichlet form.

The answer to the question is ``yes". Let $E$ be a metrizable
Lusin space, $m$  a positive $\sigma$-finite measure with full
support  on the $\sigma$-field of Borel subsets of $E$, and 
$(\EE,D(\EE))$ be a quasi-regular Dirichlet form on $L^2(E;m)$.
Our main result says that
\begin{equation}
\label{eq1.2} \MM_{0,b}(E)=L^1(E;m)+D(\EE)^*\cap\MM_b(E),
\end{equation}
i.e., if $\mu$ is a bounded Borel measure on $E$ then $|\mu|$ is
smooth with respect to the capacity determined by $(\EE,D(\EE))$
if and only if $\mu$ admits a decomposition of the form
\begin{equation}
\label{eq1.3} \mu=f\cdot m+\nu,
\end{equation}
where $f\in L^1(E;m)$, $\nu\in D(\EE)^*\cap\MM_b(E)$ and
$D(\EE)^*$ is the dual space of $D(\EE)$ equipped with the inner
product $\tilde\EE_1$ (for notation see Section \ref{sec2}).
Moreover, if $(\EE,D(\EE))$ is transient, then (\ref{eq1.2}) holds
with $D(\EE)^*$ replaced by the dual $\FF^*_e$ of the extended
Dirichlet space $(\FF_e,\tilde\EE)$. We also provide a simple
example showing that in general in  (\ref{eq1.2}) one cannot
replace $D(\EE)^*\cap\MM_b(E)$ by $(S_0-S_0)\cap\MM_b(E)$, where
$S_0$ is the set of $\EE$-smooth measures on $E$ of finite energy
(see Section \ref{sec2}) or, equivalently the set of all positive
elements of $D(\EE)^*$.

For many interesting classes of forms, one can describe the
structure of the spaces $D(\EE)^*$, $\FF^*_e$. Consequently, for
such classes the decomposition (\ref{eq1.2}) gives explicit
description of the set of bounded smooth measures. In Section
\ref{sec4}, we provide some examples to illustrate how
(\ref{eq1.2}) works in practice.

\section{Preliminaries} \label{sec2}

Throughout, we assume that $(\EE,D(\EE))$ is a quasi-regular
Dirichlet form on $L^2(E;m)$ (see \cite{CF,MOR,MR} for the
definitions). For $\alpha\ge0$ we set
$\EE_{\alpha}(u,v)=\EE(u,v)+\alpha(u,v)$, $u,v\in D(\EE)$, where
$(\cdot,\cdot)$ stands for the usual inner product in $L^2(E;m)$.
By $(\tilde\EE,D(\EE))$ we  denote the  symmetric part of
$(\EE,D(\EE))$ defined as
$\tilde\EE(u,v)=\frac12(\EE(u,v)+\EE(v,u))$, $u,v\in D(\EE)$.

We denote by $\FF_e$  the extended Dirichlet space
associated with the symmetric Dirichlet form $(\tilde\EE,D(\EE))$.
For $u\in\FF_e$ we set
$\EE(u,u)=\lim_{n\rightarrow\infty}\EE(u_n,u_n)$, where $\{u_n\}$
is an approximating sequence for $u$ (see \cite[Theorem 1.5.2]{FOT}).

If $(\EE,D(\EE))$ is transient then by \cite[Lemma 1.5.5]{FOT},
$(\FF_e,\tilde\EE)$ is a Hilbert space. Note also that if
$(\EE,D(\EE))$ is a quasi-regular Dirichlet form, then by
\cite[Proposition IV.3.3]{MR} each element $u\in D(\EE)$ admits a
quasi-continuous $m$-version denoted by $\tilde u$, and that
$\tilde u$ is $\EE$-q.e. unique for every $u\in D(\EE)$. If
moreover $(\EE,D(\EE))$ is transient, then the last statement holds
true for $D(\EE)$ replaced by $\FF_e$ (see Remark \ref{rem2.2}).

Recall that a positive measure $\mu$ on $\BB(E)$ is said to be
$\EE$-smooth ($\mu\in S$ in notation) if $\mu(B)=0$ for all
$\EE$-exceptional sets $B\in\BB(E)$ and there exists an $\EE$-nest
$\{F_k\}_{k\in\BN}$ of compact sets such that $\mu(F_k)<\infty$
for $k\in\BN$.

A measure $\mu\in S$ is said to be of finite energy integral
($\mu\in S_0$ in notation) if there is $c>0$ such that
\begin{equation}
\label{eq2.11} \int_E|\tilde u(x)|\,\mu(dx)\le c\EE_1(u,u)^{1/2},
\quad u\in D(\EE).
\end{equation}

If additionally $(\EE,D(\EE))$ is transient, then $\mu\in S$ is
said to be of finite 0-order energy integral ($\mu\in S^{(0)}_0$
in notation) if there is $c>0$ such that
\[
\int_E|\tilde u(x)|\,\mu(dx)\le c\EE(u,u)^{1/2}, \quad u\in\FF_e.
\]

If $(\EE,D(\EE))$ is regular and $E$ is a locally compact
separable metric space, then the notion of smooth measures defined
above coincides with that in \cite{FOT}. 
Moreover, if $\mu$ is a positive Radon measure on $E$ such that
(\ref{eq2.11}) is satisfied for all $v\in C_0(E)\cap D(\EE)$, then
$\mu$ charges no $\EE$-exceptional set (see \cite[Remark
A.2]{MMS}) and hence $\mu\in S_0$.

In the next section in the proof of our main theorem we will need
the lemma given below. It follows from the corresponding result
for regular forms by the so-called transfer method (see
\cite{CF,MR}) and perhaps is known, but we could not find proper
reference.

\begin{lemma}
\label{lem2.2} Assume that $(\EE,D(\EE))$ is transient. If $\mu\in
S$, then there is a nest $\{F_n\}$ such that $\fch_{F_n}\cdot\mu\in
S^{(0)}_0$ for each $n\in\BN$.
\end{lemma}
\begin{proof}
Let $(\EE^{\#},D(\EE^{\#}))$ denote the regular extension of
$(\tilde\EE,D(\EE))$ specified by \cite[Theorem VI.1.2]{MR} and
let $i:E\rightarrow E^{\#}$ denote the inclusion map. Then
$(\EE^{\#},D(\EE^{\#}))$ is transient, and by Lemma IV.4.5 and
Corollary VI.1.4 in \cite{MR}, $\mu^{\#}=\mu\circ i^{-1}$ is a
smooth measure on $\BB(E^{\#})$. Therefore, by the 0-order version
of \cite[Theorem 2.2.4]{FOT} (see remark following \cite[Corollary
2.2.2]{FOT}), there exists an $\EE^{\#}$-nest $\{F_{k}\}$ on
$E^{\#}$ such that
$\mu^{\#}_{k}\equiv\mathbf{1}_{F_{k}}\cdot\mu^{\#}\in
S^{(0)}_{0}(E^{\#})$, $k\ge 1$. Let $\{E_k\}$ be an $\EE$-nest of
\cite[Theorem VI.1.2]{MR} and let $F_{k}'=F_{k}\cap E_{k}$,
$k\in\BN$. By \cite[Corollary VI.1.4]{MR}, $\{F_{k}'\}$ is an
$\EE$-nest on $E$. Set $\mu_k=\mathbf{1}_{F_{k}'}\cdot\mu$. We are
going to show that $\mu_k\in S^{(0)}_{0}$, i.e for any nonnegative
$u\in\FF_e$,
\begin{equation}
\label{eq2.06} \langle\mu_k,\tilde u\rangle\le c\EE(u,u)^{1/2}
\end{equation}
for some $c>0$. To this end, let us consider an approximating
sequence $\{u_n\}$ for $u$ and extend $u,u_n$ to functions
$u^{\#},u_n^{\#}$ on $E^{\#}$ by putting $u^{\#}(x)=u^{\#}_n(x)=0$
for $x\in E^{\#}\setminus E$. Then
$\EE^{\#}(u^{\#}_n-u^{\#}_l,u^{\#}_n-u^{\#}_l)
=\EE(u_n-u_l,u_n-u_l)$ for $n,l\in\BN$, so $\{u^{\#}_n\}$ is
$\EE^{\#}$-Cauchy sequence. Moreover, as
$m^{\#}(E^{\#}\setminus E)=0$, $u^{\#}_n\rightarrow u^{\#}$
$m^{\#}$-a.e. Consequently, $\{u^{\#}_n\}$ is an
$\EE^{\#}$-approximating sequence for $u^{\#}$. It follows that
$u^{\#}$ belongs to the extended space $\FF^{\#}_e$ for $\EE^{\#}$
and $ \EE(u,u)=\EE^{\#}(u^{\#},u^{\#})$. Since
${\widetilde{u^{\#}}}_{|E}$ is an $m$-version of $u$ and by
\cite[Corollary VI.1.4]{MR} the function
${\widetilde{u^{\#}}}_{|E}$ is $\EE$-quasi-continuous, we have $\tilde
u={\widetilde{u^{\#}}}_{|E}$ $\EE$-q.e. From this and the fact
that $\mu_{k}^{\#}=\mu_{k}$ on $E$ it follows that
\[
\langle\mu_k,\tilde u\rangle
=\langle\mu_k,{\widetilde{u^{\#}}}_{|E}\rangle
=\langle\mu^{\#}_k,{\widetilde{u^{\#}}}\rangle \le
c\EE^{\#}(u^{\#},u^{\#})^{1/2},
\]
which gives (\ref{eq2.06}).
\end{proof}

\begin{remark}
\label{rem2.2} Note that the argument following (\ref{eq2.06})
shows that each $u\in\FF_e$ admits an $\EE$-quasi-continuous
modification.
\end{remark}

Recall that by  \cite[Theorem IV.3.5]{MR} there exists an
$m$-tight special standard Markov process $\BX=(X_t,P_x)$ properly
associated with $(\EE,D(\EE))$. By \cite[Proposition IV.2.8]{MR},
the last statement means that for all $\alpha>0$ and $f\in
L^2(E;m)$ the resolvent $(R_{\alpha})_{\alpha>0}$ of $\BX$ defined
as
\[
R_{\alpha}f(x)=E_x\int^{\infty}_0e^{-\alpha t}f(X_t)\,dt,\quad
x\in E,\,\alpha>0,\,f\in\BB^+(E)
\]
($E_x$ stands for the expectation with respect to $P_x$) is an
$\EE$-quasi-continuous version of $G_{\alpha}f$, where
$(G_{\alpha})_{\alpha>0}$ is the resolvent associated with
$(\EE,D(\EE))$.

\section{Decomposition of bounded smooth measures}

Let $\MM_{b}$ denote the set of all Borel measures $\mu$ on $E$
such that $|\mu|(E)<\infty$, where $|\mu|$ is the total variation
of $\mu$. We denote by $\MM_{0,b}$ the set of all measures
$\mu\in\MM_b$ such that $|\mu|\in S$, and by $\MM^+_{0,b}$ the
subset of $\MM_{0,b}$ consisting of all positive measures. Since
$|\mu|\in S$ if and only if $\mu$ can be expressed as
$\mu=\mu^+-\mu^-$ with $\mu^+,\mu^-\in S$, we have
$\MM_{0,b}=(S-S)\cap \MM_b$. To shorten notation, given a measure
$\mu$ on $E$ and a function $u:E\rightarrow\BR$, we write
\[
\langle\mu,u\rangle=\int_Eu(x)\,\mu(dx),
\]
whenever the integral is well defined.

In what follows we consider $D(\EE)$ (resp. $\FF_e$) to be
equipped with the scalar product $\tilde\EE_1(\cdot,\cdot)$ (resp.
$\tilde\EE(\cdot,\cdot)$). We denote by  $D(\EE)^*$ (resp. $\FF^*_e$) 
the dual space of  $D(\EE)$ (resp. $\FF^*_e$).

\begin{proposition}
\label{prop3.1} If $\mu\in D(\EE)^*\cap\MM_b$, then $|\mu|\in S$.
\end{proposition}
\begin{proof}
Since the notions of the spaces $D(\EE)^*$, $S$ only depend on the
symmetric part of the form, we may and do assume that
$(\EE,D(\EE))$ is symmetric. Let us define
$(\EE^{\#},D(\EE^{\#}))$, $E^{\#}$, $i$ as in the proof of Lemma
\ref{lem2.2} and set $\mu^{\#}=\mu\circ i^{-1}$. Then $\mu^{\#}$
is a bounded Borel measure on $\EE^{\#}$ and $\mu^{\#}\in
D(\EE^{\#})^*$. Moreover, by \cite[Corollary VI.1.4]{MR},
$|\mu|\in S$ if and only if $\mu^{\#}$ is smooth with respect to
$(\EE^{\#},D(\EE^{\#}))$. Therefore, without loss of generality, we
may and do assume that $E$ is a locally compact separable metric
space, $m$ is a positive Radon measure on $E$ with
$\mbox{supp}[m]=E$, and  $(\EE,D(\EE))$ is a regular form on
$L^2(E;m)$. Let $E=E^+\cup E^-$ be the Hahn decomposition of $E$
(for the measure $\mu$) and $B$  a Borel subset of $E$ such that
$\mbox{Cap}(B)=0$. We may assume that $B\subset E^+$. For every
$\varepsilon>0$ there exists a compact set $K\subset B$ and an
open set $U$ such that $B\subset U\subset E$ and $|\mu|(U\setminus
K)\le\varepsilon$. Let $(\EE_U,D(\EE_U))$ denote the part of
$(\EE,D(\EE))$ on $U$. Let $\eta\in D(\mathcal{E}_U)\cap C_0(U)$
be such that $\eta\ge \mathbf{1}_K$ and set
$\bar\eta=(\eta\vee0)\wedge1$. Then $\bar\eta\in
D(\mathcal{E}_U)\cap C_0(U)$, $\bar\eta\ge \mathbf{1}_K$ and
\begin{align*}
\mu^+(K)\le \int_K\bar\eta(x)\,\mu(dx)&=\int_E\bar\eta(x)\,\mu(dx)
-\int_{E\setminus K}\bar\eta(x)\,\mu(dx) \le
\|\mu\|_{D(\EE)^*}\|\bar\eta\|_{\mathcal{E}}+\varepsilon.
\end{align*}
Since $\|\bar\eta\|_{\mathcal{E}}=\|\bar\eta\|_{\mathcal{E}_U} \le
\|\eta\|_{\mathcal{E}_U}$, we have
\begin{equation}
\label{eq3.01} \mu^+(K)\le
\|\mu\|_{D(\EE)^*}\|\eta\|_{\mathcal{E}_U}+\varepsilon.
\end{equation}
Let $\mbox{Cap}$ (resp. $\mbox{Cap}_U$) denote the capacity
associated with the form $(\EE,D(\EE))$ (resp. $(\EE_U,D(\EE_U))$)
(see \cite[Section 2.1]{FOT}). Since $D(\mathcal{E}_U)\cap C_0(U)$
is a special standard core of $\EE_U$, (\ref{eq3.01})  and
\cite[Lemma 2.2.7]{FOT} imply that
\[
\mu^+(K)\le \mbox{Cap}_U(K)\cdot
\|\mu\|_{D(\EE)^*}+\varepsilon\le\varepsilon,
\]
the last inequality being a consequence of the fact that if
$\mbox{Cap}(K)=0$, then $\mbox{Cap}_U(K)=0$, which follows from
Exercise III.2.10, Theorem III.2.11(ii) and Theorem IV.5.29(i) in
\cite{MR}. Hence $\mu^+(B)\le 2\varepsilon$ for $\varepsilon>0$,
which shows that $\mu^+\in S$. In much the same way we show that
$\mu^-\in S$. Thus $|\mu|\in S$, and the proof is complete.
\end{proof}

\begin{theorem}
\label{th3.3} Assume that $(\EE,D(\EE))$ is transient and let
$\mu\in\MM_b$. Then $\mu\in\MM_{0,b}$ if and only if  there exists
$f\in L^{1}(E;m)$ and $\nu\in \FF^*_e\cap\MM_b$ such that
\begin{equation}
\label{eq3.1} \mu=f\cdot m+\nu,
\end{equation}
i.e. for every bounded $u\in\FF_e$,
\begin{equation}
\label{eq3.2} \langle\mu,\tilde u\rangle=(f,u)+\langle\nu,\tilde
u\rangle.
\end{equation}
\end{theorem}
\begin{proof}
Since the notions of the spaces  $S$, $\FF_e,\FF^{*}_e$ only
depend on  $(\tilde\EE,D(\EE))$, we may and do assume that
$(\EE,D(\EE))$ itself is symmetric.

First assume that $\mu\in\MM_{0,b}$. Without loss of generality we
may assume that $\mu$ is positive. By Lemma \ref{lem2.2} there
exists a nest $\{F_{n}\}$ such that $\mathbf{1}_{F_{n}}\cdot\mu\in
S^{(0)}_0$. Clearly, $\mu_n\equiv\mathbf{1}_{F_{n+1}\setminus
F_n}\cdot\mu\in S^{(0)}_0$ and, since $(\bigcup_{n=1}^{\infty}
F_n)^c$ is exceptional and $\mu$ is smooth,
$\mu=\sum^{\infty}_{n=1}\mu_n$. Let $\BX$ be an $m$-tight special
standard Markov process properly associated with $(\EE,D(\EE))$,
$(R_{\alpha})_{\alpha>0}$ be the resolvent of $\BX$,  and let
$A^{\mu_n}$ be a positive continuous additive functional of $\BX$
associated with $\mu_n$ in the Revuz sense (see \cite[Theorem
VI.2.4]{MR}). For $\alpha>0$ set
\[
\mu^{\alpha}_n=(\alpha U^{\alpha}_{A^{\mu_n}}1)\cdot m,
\]
where
\[
U^{\alpha}_{A^{\mu_n}}1(x)=E_x\int^{\infty}_0e^{-\alpha t}
\,dA^{\mu_n}_t,\quad x\in E.
\]
By \cite[Proposition A.7]{MMS}, $U^{\alpha}_{A^{\mu_n}}1$ is an
$\EE$-quasi-continuous version of the $\alpha$-potential
$U_{\alpha}\mu_n$ of $\mu_n$. From this and \cite[Theorem
A.8(iv)]{MMS} it follows that
\begin{equation}
\label{eq3.3}
\langle\mu^{\alpha}_n,u\rangle=\alpha(u,U^{\alpha}_{A^{\mu_n}}1)
=\alpha\langle\mu_n,R_{\alpha}u\rangle
\end{equation}
for every nonnegative Borel measurable $u$. From (\ref{eq3.3})  one can
deduce that
\begin{equation}
\label{eq3.17} \langle\mu^{\alpha}_n,u\rangle=\langle\mu_n,\alpha
R_{\alpha}u\rangle
\end{equation}
for $u\in\FF$. Let $u\in\FF_e$ and let $\{u_k\}\subset D(\EE)$ be
an approximating sequence for $u$. Then
\begin{equation}
\label{eq3.18}
\langle\mu^{\alpha}_n,u_k\rangle=\EE(U\mu^{\alpha}_n,u_k)
\le\EE(U\mu_n,U\mu_n)^{1/2}\EE(u_k,u_k)^{1/2}
\end{equation}
for every $k\in\BN$ because by (\ref{eq3.17}),
\begin{align*}
\langle\mu^{\alpha}_n,u_k\rangle=\EE(U\mu_n,\alpha R_{\alpha}u_k)
&\le\EE(U\mu_n,U\mu_n)^{1/2}\EE(\alpha
R_{\alpha}u_k,\alpha R_{\alpha}u_k)^{1/2}\\
&\le\EE(U\mu_n,U\mu_n)^{1/2}\EE(u_k,u_k)^{1/2}.
\end{align*}
Letting $k\rightarrow\infty$ in (\ref{eq3.18}), we get
\begin{equation}
\label{eq3.6}
\langle\mu^{\alpha}_n,u\rangle\le\EE(U\mu_n,U\mu_n)^{1/2}\EE(u,u)^{1/2}.
\end{equation}
Thus $\mu^{\alpha}_n\in S^{(0)}_0$. Given $\gamma\in S^{(0)}_0$,
let $T_{\gamma}$ be the bounded linear operator on
$\FF_e$ defined as
\[
T_{\gamma}(u)=\langle\gamma,\tilde u\rangle =\EE(U\gamma,u).
\]
From (\ref{eq3.6}) it follows that for every $n\in\BN$,
\[
\sup_{\alpha>0}\EE(U\mu^{\alpha}_n,U\mu^{\alpha}_n)^{1/2}
\le\sup_{\alpha>0}\|T_{\mu^{\alpha}_n}\|
\le\EE(U\mu_n,U\mu_n)^{1/2}<\infty,
\]
where $\|T_{\mu^{\alpha}_n}\|$ stands for the operator norm of
$T_{\mu^{\alpha}_n}$. By the above and the Banach-Saks theorem,
for every $n\in\BN$  we can choose  a sequence $\{\alpha^n_l\}$
such that $\alpha^n_l\rightarrow\infty$ as $l\rightarrow\infty$
and the sequence $\{U(\gamma_k(\mu_n))\}$, where
\[
\gamma_k(\mu_n)=f_k(\mu_n)\cdot m,\quad
f_k(\mu_n)=\frac1k\sum^k_{l=1}\alpha^n_lU^{\alpha^n_l}_{A^{\mu_n}}1,
\]
is $\EE$-convergent to some $g\in\FF_e$ as $k\rightarrow\infty$.
Equivalently, $\|T_{\gamma_k(\mu_n)}-T\|\rightarrow0$ as
$k\rightarrow\infty$, where $T(u)=\EE(g,u)$ for $u\in\FF_e$. On
the other hand, by (\ref{eq3.17}) and \cite[Theorem I.2.13]{MR},
for every $u\in\FF_e$,
\[
T_{\mu^{\alpha}_n}(u)=\langle\mu_n,\alpha
R_{\alpha}u\rangle=\EE(U\mu_n,\alpha
R_{\alpha}u)\rightarrow\EE(U\mu_n,u)=T_{\mu_n}(u)
\]
as $\alpha\rightarrow\infty$. It follows that in fact
$T=T_{\mu_n}$. We can therefore find a subsequence $\{k_n\}$ such
that
\begin{equation}
\label{eqa.2} \|T_{\gamma_{k_n}(\mu_{n})}-T_{\mu_n}\|\le 2^{-n}
\end{equation}
for every $n\in\BN$. Set
\[
f=\sum_{n=1}^{\infty}f_{k_n}(\mu_{n}),\quad
\nu=\sum_{n=1}^{\infty}(\mu_{n}-\gamma_{k_n}(\mu_{n})). \quad
\]
Then
\[
\mu=\sum_{n=1}^{\infty}\mu_{n}=f\cdot m+\nu.
\]
Since $m$ is $\sigma$-finite, there exists a sequence $\{U_l\}$ of
Borel subsets of $E$ such that $\bigcup^{\infty}_{l=1}U_l=E$,
$U_l\subset U_{l+1}$ and $m(U_l)<\infty$, $l\in\BN$. By
(\ref{eq3.3}),
\[
(\alpha U^{\alpha}_{A^{\mu_n}}1,\fch_{U_l}) \le\langle\mu_n,\alpha
R_{\alpha}1\rangle\le\langle\mu_n,1\rangle
\]
for every $\alpha>0$. Therefore,  for every $l\in\BN$ we have
\[
(f,\mathbf{1}_{U_l})\le \sum_{n=1}^{\infty}
(f_{k_n}(\mu_{n}),\mathbf{1}_{U_l}) \le
\sum_{n=1}^{\infty}\langle\mu_{n},1\rangle
\le\sum_{n=1}^{\infty}\|\mu_{n}\|_{TV}=\|\mu\|_{TV}.
\]
Hence $\|f\|_{L^1(E;m)}<\infty$ by the monotone convergence
theorem.  It follows in particular that $\nu=\mu-f\cdot
m\in\MM_{b}$. On the other hand, by (\ref{eqa.2}),
$\nu\in\FF^*_e$, which proves that $\mu$ is of the form
(\ref{eq3.1}).

Now, suppose that $\mu$ is given by the right-hand side of
(\ref{eq3.1}). Since $f\cdot m\in\MM_{0,b}$, we have only to prove
that if $\mu\in\FF^*_e\cap\MM_b$ then $|\mu|\in S$. But this
follows from Proposition \ref{prop3.1} since $\FF^*_e\subset
D(\EE)^*$.
\end{proof}

\begin{corollary}
\label{cor3.3} Let $\mu\in\MM_{b}$. Then $\mu\in\MM_{0,b}$ if and
only if there exist $f\in L^1(E;m)$ and $\nu\in D(\EE)^*$ such that
\mbox{\rm(\ref{eq3.2})} holds true for every bounded $u\in
D(\EE)$.
\end{corollary}
\begin{proof}
The Dirichlet form  $(\EE_1,D(\EE))$ is  transient, quasi-regular
and its extended Dirichlet space is $(D(\EE),\tilde\EE_1)$.
Moreover, $|\mu|$ is smooth with respect to $(\EE,D(\EE))$ if and
only if it is smooth with respect to $(\EE_1,D(\EE))$. Therefore
the corollary follows from Theorem \ref{th3.3} applied to the form
$(\EE_1,D(\EE))$.
\end{proof}

\begin{corollary}
\label{cor3.4} Let $\mu\in\MM_{b}$. Then $\mu\in\MM_{0,b}$ if and
only if there exist $f\in L^{1}(E;m)$ and $v\in D(\EE)$  such that
for every bounded $u\in D(\EE)$,
\begin{equation}
\langle\mu,\tilde u\rangle=(f,u)+\EE_1(v,u).
\end{equation}
\end{corollary}
\begin{proof}
Follows immediately from Corollary \ref{cor3.3} and the
Lax-Milgram theorem.
\end{proof}

\begin{remark}
(i) The decomposition (\ref{eq3.1}) is not unique since
$L^1(E;m)\cap\FF_e^{*}\neq\emptyset$.
\smallskip\\
(ii) In general, $D(\EE)^*\cap\MM_b$ in the decomposition
(\ref{eq1.2}) cannot be replaced by $(S_0-S_0)\cap\MM_b$. To see
this, let us consider the classical Dirichlet form (see Example
\ref{ex4.1} in the next section) with $D=B(0,1)\subset\BR^7$,
where $B(0,r)$ denote the open ball  with radius $r>0$ and center
at 0. Let $\sigma_{a_n}$ denote the surface measure on $\partial
B(0,a_n)$ with $a_n=n^{-1/4}$, and let
\[
\mu=\sum_{n=1}^\infty \sigma_{a_n}.
\]
From \cite[Example 5.2.2]{FOT} it follows that $\sigma_{a_n}\in S$
for each $n\in\BN$. Hence $\mu\in S$. Moreover,
$\mu(D)=c\sum^{\infty}_{n=1}a_n^{6}<\infty$, so
$\mu\in\MM^+_{0,b}$. Let
\begin{equation}
\label{eq3.7} \mu(dx)=f\,dx+\nu(dx)
\end{equation}
be a decomposition of $\mu$ as in Theorem \ref{th3.3}. Then
$\nu\notin S_0-S_0$. To show this, let us denote by $\mu_s,\nu_s$
the singular parts (with respect to the Lebesgue measure) of $\mu$
and $\nu$, respectively, and observe that $\mu=\mu_s=\nu_s$.
Suppose that $\nu\in S_0-S_0$. Then $\nu^+,\nu^-\in S_0$ and hence
$\nu^+_s,\nu^-_s\in S_0$ because $\nu^+_s\le\nu^+$ and
$\nu^-_s\le\nu^-$. Consequently,  $\mu_s=\nu^+_s-\nu^-_s\in
S_0-S_0$, which implies that $\mu\in S_0$ since $\mu$ is positive.
On the other hand, if we set $u(x)=|x|^{-2}-1$, $x\in D$, then
$u\in H^1_0(D)$
and
\[
\langle\mu,u\rangle=\sum_{n=1}^{\infty}\int_{\partial B(0,a_n)}
u(x)\,\sigma_{a_n}(dx)=c\sum^{\infty}_{n=1}(a_n^4-a_n^6)=\infty,
\]
which is a contradiction, because by \cite[Theorem 2.2.2]{FOT}, if
$\mu\in S_0$, then quasi-continuous elements of $H^1_0(D)$ are
integrable with respect to $\mu$.
\smallskip\\
(iii) From the proof of Theorem \ref{th3.3} it follows that if
$\mu\in\MM^+_{0,b}$, then the $L^1$ part $f$ of its decomposition
can be chosen to be positive. The example given above shows that
in general this is not true  for $\nu$, because if $\nu$ in
(\ref{eq3.7}) were positive, we would have $\nu\in S_0$ and hence
$\mu\in S_0$.
\smallskip\\
(iv) By the definition of $S_0$,  $(S_0-S_0)\cap\MM_b\subset
D(\EE)^*\cap\MM_b$. The opposite inclusion is false and  $\nu$ of
(\ref{eq3.7}) can serve as a counterexample. Below we give an
explicit construction of another counterexample. Let $D$, $a_n$,
$\sigma_{a_n}$ be  as in (ii), and let
$b_n=(\frac{n+(1/2)}{n(n+1)})^{1/4}$, so that
$a_1>b_1>a_2>b_2>\cdots$. Let $\nu_{a_n}(x)$, $\nu_{b_n}(x)$ denote
the first components of the outer normal vectors to $\partial
B(0,a_n)$ and $\partial B(0,b_n)$ at  $x$. Set
\[
\mu(dx)=\sum_{n=1}^\infty
(\nu_{a_n}(x)\,\sigma_{a_n}(dx)-\nu_{b_n}(x)\,\sigma_{b_n}(dx)).
\]
Then $\mu\in\mathcal{M}_{0,b}$ and for every $\eta\in H^1_0(D)$ we
have
\begin{align*}
\langle\mu,\eta\rangle&=\sum_{n=1}^\infty\Big(\int_{\partial
B(0,a_n)}\eta(x)\nu_{a_n}(x)\,\sigma_{a_n}(dx)-\int_{\partial
B(0,b_n)}\eta(x)\nu_{b_n}(x)\,\sigma_{b_n}(dx)\Big)\\&
=\sum_{n=1}^\infty \int_{B(0,a_n)\setminus
B(0,b_n)}\frac{\partial\eta}{\partial x_1}(x)\,dx\le C
\|\eta\|_{H^1_0(D)}.
\end{align*}
Hence $\mu\in H^{-1}(D)$. But $|\mu|\notin H^{-1}(D)$ because if
it were true, the series
\[
\sum_{n=1}^{\infty}\int_{\partial B(0,a_n)} u(x)\,\sigma_{a_n}(dx)
=c\sum^{\infty}_{n=1}a_n^4
\]
would be convergent.
\end{remark}

\section{Examples} \label{sec4}

In this section, we apply Theorem \ref{th3.3} to give explicit
description of the set $\MM_{0,b}$ for some classes of regular
local forms, regular nonlocal forms and quasi-regular forms.

\begin{example}
\label{ex4.1} (Classical Dirichlet form) Let $D\subset\BR^d$ be a
bounded domain. Consider the classical form
\[
\BD(u,v)=\frac12\int_D\langle\nabla u,\nabla
v\rangle_{\BR^d}\,dx,\quad u,v\in D(\EE)=H^1_0(D).
\]
It is known that $(\BD,H^1_0(D))$ is a transient regular Dirichlet
form on $L^2(D;dx)$ (see \cite[Example 1.5.1]{FOT}). If
$\mu\in\MM^+_b$, then $\mu\in S$ if and only if $\mu$ charges no
set of Newtonian capacity zero. By Poincar\'e's inequality, the
norms determined by $\BD$ and $\BD_1$ are equivalent (of course,
the norm determined by $\BD_1$ is the usual norm in the Sobolev
space $H^1_0(D)$). As a consequence, $\FF_e=H^1_0(D)$ and hence
$\FF^*_e=H^{-1}(D)$. From Theorem \ref{th3.3} and the well known
characterization of $H^{-1}(D)$ it follows that if $\mu\in\MM_{b}$,
then $\mu\in\MM_{0,b}$ if and only if
\begin{equation}
\label{eq4.1} \mu=f^0-\dyw F
\end{equation}
for some $f^0\in L^1(D;dx)$ and $F=(F^1,\dots,F^d)$ such that
$F^i\in L^2(D;dx)$, $i=1,\dots,d$. We see that in the case of the
classical form the decomposition of Theorem \ref{th3.3} reduces to
the decomposition proved in \cite[Theorem 2.1]{BGO}.
\end{example}

On can easily check that (\ref{eq4.1}) holds for more general
(possibly nonsymmetric) regular Dirichlet forms defined by Eq.
(2.17) in \cite[Section II.2]{MR} with coefficients satisfying the
assumptions of  \cite[Proposition II.2.11]{MR}.

\begin{example}
(Regular nonlocal Dirichlet forms) Let $\psi:\BR^d\rightarrow\BR$
be a continuous negative definite function, let $s\in\BR$, and let
\[
H^{\psi,s}=\{u\in\SSS'(\BR^d):\int_{\BR^d}(1+\psi(\xi))^s |\hat
u(\xi)|^2\,d\xi <\infty\},
\]
where $\SSS'(\BR^d)$ is the space of tempered distributions  and
$\hat u$ is the Fourier transform of $u$. Note that
$H^{\psi,1}=\{u\in L^2(\BR^d;dx):\int_{\BR^d}\psi(\xi) |\hat
u(\xi)|^2\,d\xi <\infty\}$. Consider the form
\[
\Psi(u,v)=\int_{\BR^d}\psi(\xi)\hat u(\xi)\overline{\hat
v(\xi)}\,d\xi,\quad u,v\in H^{\psi,1}.
\]
It is known (see \cite[Example 1.4.1]{FOT}) that
$(\Psi,H^{\psi,1})$ is a symmetric regular Dirichlet form on
$L^2(\BR^d;dx)$.

By \cite[Theorem 3.10.11]{J}, the dual space of $H^{\psi,1}$ is the
space $H^{\psi,-1}$ in the sense that for every
$\nu\in(H^{\psi,1})^*$ there is $v\in H^{\psi,-1}$ such that $\hat
v\hat u$ may be interpreted as an element of $L^1(\BR^d;dx)$ and
the value of $\nu$ on $u\in H^{\psi,1}$ is equal to
$\int_{\BR^d}\hat v(x)\hat u(x)\,dx$. Therefore from Corollary
\ref{cor3.3} it follows that if $\mu\in\MM_{0,b}$, then there exist
$f\in L^1(\BR^d;dx)$ and $v\in\ H^{\psi,-1}$ such that for every
bounded $u\in H^{\psi,1}$,
\[
\langle\mu,\tilde u\rangle=(f,u)+\int_{\BR^d}\hat v(x)\hat
u(x)\,dx.
\]
Since $H^{\psi,1}=H^1(\BR^d)$ for $\psi(\xi)=|\xi|^2$, the above
decomposition may be viewed as a generalization of (\ref{eq4.1}).
\end{example}

\begin{example}
\label{ex4.3} (Gradient Dirichlet forms on infinite dimensional
space) Let $H$ be a separable real Hilbert space and let $A$ be a
self-adjoint operator $A$ such that $\langle
Ax,x\rangle_H\le-\omega|x|_H^2$, $x\in D(A)$, for some $\omega>0$
and $A^{-1}$ is of trace class. Let $Q_{\infty}=-\frac12 A^{-1}$,
and let $\gamma$ denote the Gaussian measure on $H$ with mean 0
and covariance operator $Q_{\infty}$. We consider the form
\[
\EE(u,v)=\frac{1}{2}\int_H\langle\nabla u,\nabla
v\rangle_H\,\gamma(dx), \quad u,v\in\FF C^{\infty}_b,
\]
where $\FF C^{\infty}_b$ is the space of finitely based smooth
bounded functions on $H$ (see \cite[Section II.3]{MR} and $\nabla$
is the $H$-gradient defined for $u\in\FF C^{\infty}_b$ as the
unique element of $H$ such that $\langle\nabla
u(x),h\rangle_H=\frac{\partial u}{\partial h}(x)$ for $x\in H$. By
\cite[Proposition II.3.8]{MR}, the form $(\EE,\FF C^{\infty}_b)$ is
closable and its closure, which we denote by $(\EE,W^{1,2}(H))$,
is a symmetric Dirichlet form. Moreover, by results of
\cite[Section IV.4]{MR}, it is quasi-regular.

By Corollary \ref{cor3.4}, if $\mu\in\MM_{b}$, then
$\mu\in\MM_{0,b}$ if and only if there exist $f\in L^1(H;\gamma)$ and
$v\in W^{1,2}(H)$ such that for every bounded $u\in W^{1,2}(H)$,
\begin{equation}
\label{eq4.2} \langle\mu,\tilde u\rangle=\int_H fu\,\gamma(dx)
+\EE_1(v,u).
\end{equation}
In fact, $\mu\in\MM_{0,b}$ if and only if for some $f^0\in
L^1(H;\gamma)$ and $F\in L^2(H;\gamma)$ it can be written in the form
\begin{equation}
\label{eq4.3} \mu=f^0-\dyw_{\gamma}F,
\end{equation}
similar to (\ref{eq4.1}). Indeed, if $\mu\in\MM_{0,b}$, then by
(\ref{eq4.2}),
\begin{equation}
\label{eq4.4} \langle\mu,\tilde u\rangle=\int_Hf^0u\,\gamma(dx)
+\frac12\int_H\langle F,\nabla u\rangle_H\,\gamma(dx)
\end{equation}
with $f^0=v+f\in L^1(H;\gamma)$ and $F=\nabla v\in L^2(H;\gamma)$. On
the other hand, if $F\in C^1_b(H;H)$ has finite divergence with
respect to $\gamma$ (see \cite[Section 11.1]{DPZ} for the
definition), then by \cite[Lemma 11.1.9]{DPZ},
\begin{equation}
\label{eq4.5} \int_H\langle F,\nabla u\rangle_H\,\gamma(dx)
=-\int_H(\dyw_{\gamma}F)u\,\gamma(dx),
\end{equation}
where $\dyw_{\gamma}F(x)=\dyw F(x)-\langle
Q_{\infty}^{-1}x,F(x)\rangle_H$, $x\in H$, which when combined
with (\ref{eq4.4}) makes it legitimate to write $\mu$ in the form
(\ref{eq4.3}). Conversely, if $\mu\in\MM_b$ and $\mu$ is of the
form (\ref{eq4.3}) for some $f^0\in L^1(H;\gamma)$ and $F\in
L^2(H;\gamma)$, then $\dyw_{\gamma}F\in\MM_b$ and by (\ref{eq4.5}),
\[
|\langle \dyw_{\gamma}F,u\rangle|\le\|F\|_{L^2(H;\gamma)}\|\nabla
u\|_{L^2(H;\gamma)}\le C\EE(u,u)^{1/2}
\]
for all bounded $u\in W^{1,2}(H)$. That $\dyw_{\gamma}F$ is smooth
now follows from \cite[Lemma 2.2.3]{FOT}.
\end{example}

The assertion that $\mu\in\MM_{0,b}$ if and only if $\mu$ has a
decomposition (\ref{eq4.3}) holds true for forms more general than
those considered in Example \ref{ex4.3}. In fact, slightly modifying
the argument in Example  \ref{ex4.3}, one can show that it
holds for the form which is the closure of the form defined by
\cite[(20)]{F} if \cite[Hypotheses 3.1 and 3.2]{F} (the latter with $R=R^*>0$
such that $R^{-1}$ is bounded)  are satisfied.

\subsection*{Acknowledgements}

The first author was supported by Polish National Science Centre (grant  no.\\
2012/07/D/ST1/02107.

\noindent Tomasz Klimsiak, Andrzej Rozkosz\\
Faculty of Mathematics and Computer Science\\
Nicolaus Copernicus University\\
Chopina 12/18\\
87-100 Toru\'n, Poland\\
E-mail: tomas@mat.umk.pl, rozkosz@mat.umk.pl

\end{document}